\documentclass[english,letterpaper]{amsart}
\usepackage[english]{babel}
\usepackage{amsmath,amsthm,amssymb}
\usepackage{latexsym}
\usepackage{enumerate}
\usepackage[latin1]{inputenc}
\usepackage{layout}
\usepackage{type1cm}
\usepackage{hyperref}

\usepackage{color}

\newtheorem{theorem}{Theorem}
\newtheorem{lemma}{Lemma}
\newtheorem{proposition}{Proposition}

\newtheorem{remark}{Remark}

\newcommand{\victor}[1]{{\color{black} #1}}

\title[Involution kernel and L.D.P. on $\beta$-shifts]{On involution kernels and large deviations principles on $\beta$-shifts}

\author[Victor Vargas]{}



\email{vavargascu@gmail.com}

\thanks{Supported by FFJC-MINCIENCIAS Process 80740-628-2020.}

\begin{document}
\maketitle

\centerline{\scshape Victor Vargas}
\medskip
{\footnotesize
 \centerline{School of Mathematics - National University of Colombia}
   \centerline{Medell\'in - Colombia}
}

\begin{abstract}
Consider $\beta > 1$ and $\lfloor \beta \rfloor$ its integer part. It is widely known that any real number $\alpha \in \Bigl[0, \frac{\lfloor \beta \rfloor}{\beta - 1}\Bigr]$ can be represented in base $\beta$ using a development in series of the form $\alpha = \sum_{n = 1}^\infty x_n\beta^{-n}$, where $x = (x_n)_{n \geq 1}$ is a sequence taking values into the alphabet $\{0,\; ...\; ,\; \lfloor  \beta \rfloor\}$. The so called $\beta$-shift, \victor{denoted by $\Sigma_\beta$,} is given as the set of sequences such that all their iterates by the shift map are less than or equal to the quasi-greedy $\beta$-expansion of $1$. Fixing a H\"older continuous potential $A$, we show an explicit expression for the main eigenfunction of the Ruelle operator $\psi_A$, \victor{in order to obtain a natural extension to the bilateral $\beta$-shift of its corresponding Gibbs state $\mu_A$.} Our main goal here is to prove a first level large deviations principle for the family $(\mu_{tA})_{t>1}$ with a rate function $I$ attaining its maximum value on the union of the supports of all the maximizing measures of $A$. \victor{The above is proved through a technique using the representation of $\Sigma_\beta$ and its bilateral extension $\widehat{\Sigma_\beta}$ in terms of the quasi-greedy $\beta$-expansion of $1$ and the so called involution kernel associated to the potential $A$.} 
\end{abstract}

\section{Introduction}

The theory of large deviations arises from the study of the so called rare events modeled by stochastic phenomena. More specifically, from a probabilistic point of view, this theory is concerned with the study of rates of convergence (in an exponential scale), of the tails associated to probabilistic distributions given by the empirical averages of families of independent identically distributed random variables (see for details \cite{MR2571413} and \cite{MR2189669}). Some of the most captivating introductory works about this theory appear in \cite{MR203230} and the series of papers beginning with \cite{MR386024}, in which are introduced classical definitions of the theory of large deviations from a probabilistic flavor. In the matter of statistical mechanics, the theory of large deviations appears as a useful tool to understand exponential rates of convergence at zero temperature for families of equilibrium states associated to potentials satisfying suitable conditions of regularity. In the case when the so called rate function is defined on \victor{the space on which the dynamical system acts,} the large deviations principle is known as a first level one. \victor{On the other hand,} when the rate function \victor{depends on the Borel probability} measures defined on the space on which \victor{the dynamical system acts,} the large deviations principle is known as a second level one. \victor{There are several works where those problems were studied assuming another hypothesis on the dynamics. In particular, some interesting ones about first level large deviations principles in the contexts of Markov shifts and XY models were presented in \cite{MR3114331}, \cite{MR2210682}, \cite{BMP15} and \cite{MR3170635}. Here we are interested in the matter of $\beta$-shifts.}   

Given $\beta > 1$, the so called $\beta$-shift arises in the classical mathematical literature as the typical example of a dynamical system with topological entropy equal to $\log(\beta)$. This class of compact subshifts satisfies interesting dynamical properties and can be characterized through $\beta$-expansions of real numbers which, in particular, gives an attractive relation between number theory and symbolic dynamics. Classical contributions about this class of subshifts were presented by \victor{Erd\"os et al.}, Parry and R\'enyi in the works \cite{MR1078082}, \cite{MR0142719}, and \cite{MR0097374}. Furthermore, in this matter \victor{also have been presented fascinating contributions in the area of thermodynamical formalism.} For instance, in \cite{MR0466492} was proved existence of equilibrium states using a Ruelle operator defined on the class of $\beta$-transformations. Some time later, in \cite{MR2109476} was proved a second level large deviations principle on $\beta$-shifts without satisfying the specification property (as a particular case of dynamical \victor{systems with the} approximate product property). The foregoing, using a rate function defined on the space of \victor{Borel probability measures} and depending of the metric entropy and a lower-energy function. Subsequently, in \cite{MR3624405} was presented a second level large deviations principle using a rate function depending of the metric pressure and defined on a wide class of compact subshifts satisfying a non-uniform structure, among them, as a particular case, the so called $\beta$-shifts. 

\victor{In this work, we want to prove} a first level large deviations principle on the setting of $\beta$-shifts. The best of our knowledge is to use the characterization of the $\beta$-shift through the \victor{quasi-greedy $\beta$-expansion of the number $1$} and the so called involution kernel to obtain our main result. \victor{As far as we know, this is the first work that presents a first level large deviations principle (without using a second level one), in a context beyond the Markovian approach into the class of subshifts defined on finite alphabets. In order to do that, fixing} a potential $A$ satisfying suitable conditions, we define a Ruelle operator $L_A$ and an involution kernel $W$ and we present an explicit expression of the main eigenfunction $\psi_A$ of the Ruelle operator. Besides that, we find extensions of each one of the elements of the family of Gibbs states $(\mu_{tA})_{t>1}$ to the bilateral $\beta$-shift which satisfy interesting properties in their marginals \victor{that allow a simplification of the main problem through an extension of the rate function to the bilateral $\beta$-shift.} As a conclusion, we present a proof of a sort of large deviations principle on a family of functions converging pointwise to the rate function. It is important to point out that the rate function is defined using calibrated sub-actions associated to the potential $A$ and, as an interesting property, the so called rate function attains its maximum value on the union of the supports of all the maximizing measures.

The paper is organized as follows:

In section \ref{preliminaries-section} are presented the main definitions that will be used \victor{throughout} the paper, among them, the definition of Ruelle operator and its corresponding dual. In section \ref{involution-kernel-section} is introduced an adapted definition of involution kernel for the matter of $\beta$-shifts and are proved some technical results about the behavior of the transpose potential and the main eigenfunctions of the Ruelle operators associated to both, the potential and its corresponding transpose. In section \ref{large-deviations-section} are proved some technical lemmas and a first level large deviations principle, which is the main result of the paper.

\section{Preliminaries}
\label{preliminaries-section}

Consider a real number $\beta > 1$ and denote by $\lfloor \beta \rfloor$ its corresponding integer part. Fixing the alphabet $\mathcal{A}_\beta := \{0,\; ...\; ,\; \lfloor \beta \rfloor\}$, we denote by $\Sigma_{\lfloor \beta \rfloor}$ the set of sequences taking values into $\mathcal{A}_\beta$. For each $n \in \mathbb{N}$, we let \victor{$\mathcal{A}_\beta^n$ be the set of words of length $n$ on the alphabet $\mathcal{A}_\beta$.} Note that both of these sets can be equipped with the lexicographic order $\prec$ and both of them result in total ordered sets. 

Let $\sigma:\; \Sigma_{\lfloor \beta \rfloor} \to \Sigma_{\lfloor \beta \rfloor}$ be the {\bf shift map} given by $\sigma((x_n)_{n \geq 1}) := (x_{n+1})_{n \geq 1}$. \victor{Throughout the paper,} we call {\bf subshift} to any closed $\sigma$-invariant set $Y$ contained in $\Sigma_{\lfloor \beta \rfloor}$ with the shift map $\sigma$ acting on it. In the particular case when $Y = \Sigma_{\lfloor \beta \rfloor}$ \victor{we call $Y$ a {\bf full-shift}.} If there is no confusion, we will also denote by $\sigma$ the corresponding restriction of the shift map to the set $Y$. We say that \victor{a word $w = w_1\; ...\; w_n \in \mathcal{A}_\beta^n$ is {\bf admissible}} in the subshift $Y$, if there is a sequence $x \in Y$ containing the word $w$, that is, there exists some $k \in \mathbb{N}$ such that $w = x_k\; ...\; x_{k + n -1}$. In such a case, we say that the word $w$ is a {\bf factor} of the sequence $x$. 

\victor{Throughout the paper,} we denote by $\mathcal{L}_n(Y)$ the set of admissible words of length $n$ in $Y$. We call {\bf language} of $Y$, and denote by $\mathcal{L}(Y)$, to the set of all the admissible words in $Y$. Actually, it is not difficult to check that $\mathcal{L}(Y) = \bigcup_{n=1}^{\infty}\mathcal{L}_n(Y)$. Given $w = w_1\; ...\; w_n \in \mathcal{L}(Y)$, we define the {\bf cylinder} associated to the word $w$ as the set $[w] := \{x \in Y :\; x_1 = w_1,\; ...\; ,\; x_n = w_n\}$ and we use the notation $|w|$ for the length of the word $w$. We say that $x = (x_n)_{n \geq 1} \in \Sigma_{\lfloor \beta \rfloor}$ is a {\bf finite sequence}, if there is $n_0 \in \mathbb{N}$ such that $x_n = 0$ for each $n \geq n_0$, \victor{in the other case,} we say that $x$ is an {\bf infinite sequence}. 

Given a sequence $x \in \Sigma_{\lfloor \beta \rfloor}$, we denote by $[x]_k$ the cylinder associated to the word $x_1\; ...\; x_k$ and we use the notation $\beta_x$ for the unique real number generated by the sequence $x$ in the base $\beta$, that is, \victor{the only one satisfying} 
\[
\beta_x := \sum_{n = 1}^\infty x_n\beta^{-n} \;.
\]  

Consider a number $\alpha \in \Bigl[0, \frac{\lfloor \beta \rfloor}{\beta - 1}\Bigr]$, we say that a sequence $(x_n)_{n \geq 1} \in \Sigma_{\lfloor \beta \rfloor}$ is a {\bf $\beta$-expansion} of the number $\alpha$, when \victor{it satisfies} the equation 
\[
\alpha = \sum_{n = 1}^{\infty} x_n \beta^{-n} \;.
\]

It is widely known that the set of $\beta$-expansions associated to a real number $\alpha$ can be even \victor{an infinite set (see for instance \cite{MR1153488}). Moreover, it was proved that for any $\beta \in (1, 2)$, almost any number $\alpha \in \Bigl[0, \frac{1}{\beta - 1}\Bigr]$ has a continuum of $\beta$-expansions (see for details \cite{zbMATH02074190}).} Because of that, \victor{we call} {\bf greedy} $\beta$-expansion of a number $\alpha$ to the largest $\beta$-expansion of $\alpha$ with respect to the lexicographic order $\prec$. Observe that the greedy $\beta$-expansion can be either a finite or an infinite sequence. By the above, we call {\bf quasi-greedy} $\beta$-expansion of the number $\alpha$ to the largest infinite $\beta$-expansion of $\alpha$. It is easy to check that the greedy $\beta$-expansion always is greater than or equal to the quasi-greedy $\beta$-expansion, with equality only when the greedy $\beta$-expansion is infinite (see for details \cite{MR1020481, MR0142719}). 

Given $a \in \mathcal{A}_\beta$ and a sequence $x = (x_n)_{n \geq 1} \in \Sigma_{\lfloor \beta \rfloor}$, we will use the notation $a^{\infty} \in \Sigma_{\lfloor \beta \rfloor}$ for the constant sequence taking only the value $a \in \mathcal{A}_\beta$ and the notation $ax$ for the sequence $(a, x_1, x_2,\; ...\;) \in \Sigma_{\lfloor \beta \rfloor}$.

Now we are able to present a formal definition of the class of subshifts on which we develop the theory \victor{throughout the paper,} widely known in the mathematical literature as $\beta$-shifts (see for details \cite{MR1020481}, \cite{MR878240} and \cite{MR648108}). \victor{Let us denote by $x^\beta$ the quasi-greedy $\beta$-expansion of $1$. Then, the so called {\bf $\beta$-shift} is defined as 
\begin{equation}
\label{beta-shift}
\Sigma_\beta := \{x \in \Sigma_{\lfloor \beta \rfloor} :\; \sigma^k x \preceq x^\beta, \forall k \in \mathbb{N} \cup \{0\}\} \,,
\end{equation}
with the shift map acting on it (see for details \cite{MR1078082}).} Moreover, \victor{as a consequence} of the greedy algorithm, which was proposed in \cite{MR0142719} and \cite{MR1078082}, we can characterize $\Sigma_\beta$ through admissible words in the following way: 
\begin{equation}
\label{greedy-algorithm}
\Sigma_\beta = \Bigl\{x \in \Sigma_{\lfloor \beta \rfloor} :\; x_k\; ...\; x_{k+m-1} \prec x^\beta_1\; ...\; x^\beta_m,\; \forall k, m \in \mathbb{N} \Bigr\} \;.
\end{equation}

\victor{Furthermore, since $\sigma^k(x^\beta) \prec x^\beta$ for each $k \in \mathbb{N}$, it follows that the quasi-greedy $\beta$-expansion of $1$ belongs to the set in the right side of \eqref{greedy-algorithm}, guaranteeing that the sets in \eqref{beta-shift} and \eqref{greedy-algorithm} agree (see for details \cite{MR1020481}, \cite{MR1851269} and \cite{MR0097374}). On the other hand, it is well known that the subshift $\Sigma_\beta$ is a finite type one if, and only if, the greedy $\beta$-expansion of $1$ is a finite sequence (see for details \cite{MR1020481} and \cite{MR0142719}).} 

It is widely known that any subshift $Y \subset \Sigma_{\lfloor \beta \rfloor}$ can be equipped with the metric
\begin{equation}
\label{metric}
d(x, y) := 2^{-\min\{n \in\mathbb{N} :\; x_n \neq y_n\} + 1} \,, 
\end{equation}
and results in a compact metric space. \victor{Actually, the topology induced by \eqref{metric} agrees with the subspace topology with respect to the product topology on the full-shift $\Sigma_{\lfloor \beta \rfloor}$. In particular, the subshift $\Sigma_\beta$} results in a compact metric space when it is equipped with the metric in \eqref{metric}. 

\victor{A subshift $Y \subset \Sigma_{\lfloor \beta \rfloor}$ is called {\bf topologically transitive}, if for any pair of cylinders $[u], [v] \subset Y$, there is $n \in \mathbb{N} \cup \{0\}$ such that $\sigma^{-n}([u]) \cap [v] \neq \emptyset$. In the case that the set $\{n :\; \sigma^{-n}([u]) \cap [v] \neq \emptyset\}$ is co-finite in $\mathbb{N} \cup \{0\}$ for any pair of cylinders $[u], [v] \subset Y$, the subshift $Y \subset \Sigma_{\lfloor \beta \rfloor}$ is called {\bf topologically mixing}.} Besides that, if there is $n_0 \in \mathbb{N} \cup \{0\}$ such that \victor{for any $u, v \in \mathcal{L}(Y)$, there is a word $w \in \mathcal{A}_\beta^{n_0}$} such that the word $uwv \in \mathcal{L}(Y)$, we say that the subshift $Y \subset \Sigma_{\lfloor \beta \rfloor}$ satisfies the {\bf specification property}. 

It is not difficult to check that the specification property implies the topologically mixing property (see for details \cite{zbMATH03514123} and \cite{MR1484730}). \victor{In particular, for the matter of the so called $\beta$-shifts, existence of $n_0 \in \mathbb{N}$ such that all the words composed only by $0$'s in the quasi-greedy $\beta$-expansion of $1$ have length less than or equal to $n_0$, is equivalent to say that $\Sigma_\beta$ satisfies the specification property (see for details \cite{MR1020481} and Theorem $2$ in \cite{MR878240}). In particular, under the conditions mentioned above $\Sigma_\beta$ results in a topologically mixing compact subshift on the alphabet $\{0,\; ... \;,\; \lfloor \beta \rfloor \}$.} 

Let $Y \subset \Sigma_{\lfloor \beta \rfloor}$ be a subshift, we use the notation $\mathcal{C}(Y)$ for the set of continuous functions from $Y$ into $\mathbb{R}$ and we denote by $\mathcal{H}_\theta(Y)$, $\theta \in (0, 1]$, the set of $\theta$-H\"older continuous functions from $Y$ into $\mathbb{R}$. Consider $\varphi \in \mathcal{H}_\theta(Y)$, then, its corresponding {\bf H\"older constant} is given by
\[
\mathrm{Hol}_\varphi := \sup\Bigl\{ \frac{|\varphi(x) - \varphi(y)|}{d(x, y)^\theta} :\; x \neq y \Bigr\} \;.
\]

Fixing a constant $K > 0$, we define the set \victor{
\begin{equation}
\label{convex-set}
\Lambda_K(Y) := \Bigl\{\varphi \in \mathcal{C}(Y) :\; 0 < \varphi \leq 1,\; \frac{\varphi(x)}{\varphi(y)} \leq e^{K d(x, y)^{\theta}}, \text{ when, } ax \in Y \Leftrightarrow ay \in Y \Bigr\} \;.
\end{equation}}

It is not difficult to check that for each $K > 0$, the set $\Lambda_K(Y)$ is a convex closed subset of $\mathcal{C}(Y)$ \victor{(see for instance Lemma $3$ in \cite{MR1860762} and Theorem $2.2$ in \cite{MR1085356}).}

Given a subshift $Y \subset \Sigma_{\lfloor \beta \rfloor}$, we denote by $\mathcal{B}(Y)$ the set of finite \victor{Borel measures} on $Y$, by $\mathcal{P}(Y)$ the set of \victor{Borel probabilities} on $Y$ and by $\mathcal{P}_\sigma(Y)$ the set of \victor{Borel $\sigma$-invariant probabilities} on $Y$. In addition, given a potential $A \in \mathcal{C}(Y)$ we define the {\bf maximizing value} of $A$ as  
\begin{equation}
\label{maximizing-value}
m(A) := \sup \Bigl\{ \int_Y A d\mu :\; \mu \in \mathcal{P}_\sigma(Y) \Bigr\} \;.
\end{equation}

The set $\mathcal{P}_{\max}(A)$ of {\bf maximizing measures} is defined as  the set of $\sigma$-invariant probabilities that attain the maximizing value of $A$. That is, the ones attaining the $\sup$ in \eqref{maximizing-value}. Therefore, we have $\mu \in \mathcal{P}_{\max}(A)$ if, and only if, $\int_Y A d\mu = m(A)$. Under the assumptions that appear above, we say that a map $V \in \mathcal{C}(Y)$ is a {\bf calibrated sub-action} associated to the potential $A$ when \victor{it satisfies that}
\begin{equation}
\label{calibrated-sub-action}
m(A) = \sup\{A(z) + V(z) - V(x) :\; \sigma(z) = x \} \;.
\end{equation}

On the other hand, if $X, Y \subset \Sigma_{\lfloor \beta \rfloor}$ are subshifts, $\mu \in \mathcal{P}(X)$ and $\nu \in \mathcal{P}(Y)$, we denote the {\bf transport plan} between $\mu$ and $\nu$ by $\Gamma(\mu, \nu)$ \victor{which is defined} as the set of measures $\Pi \in \mathcal{P}(X \times Y)$ such that $\Pi(E \times Y) = \mu(E)$ and $\Pi(X \times F) = \nu(F)$ for any pair of \victor{Borel sets} $E \subset X$ and $F \subset Y$.

The so called {\bf Ruelle operator} associated to a potential $A \in \mathcal{H}_\theta(\Sigma_\beta)$ is defined as the map assigning to each $\varphi \in \mathcal{C}(\Sigma_\beta)$ the function
\begin{equation}
\label{Ruelle-operator}
L_A(\varphi)(x) := \sum_{\sigma(z) = x}e^{A(z)}\varphi(z) = \sum_{a = 0}^{\lfloor \beta - \beta_x \rfloor}e^{A(ax)}\varphi(ax) \;.
\end{equation}

It is not difficult to check that the Ruelle operator preserves the set $\Lambda_K(\Sigma_\beta)$ for any $K > 0$ (where $\Lambda_K(\; \cdot \;)$ is defined as appears in \eqref{convex-set}). Indeed, given $\psi \in \Lambda_K(\Sigma_\beta)$ and $x, y \in \Sigma_\beta$ such that $\lfloor \beta - \beta_x \rfloor = \lfloor \beta - \beta_y \rfloor$, we have 
\begin{align*}
L_A(\psi)(x) 
&= \sum_{a = 0}^{\lfloor \beta - \beta_x \rfloor}e^{A(ax)}\varphi(ax) \\
&\leq e^{2^{-1}Kd(x, y)^{\theta}}\sum_{a = 0}^{\lfloor \beta - \beta_y \rfloor}e^{A(ay)}\varphi(ay) \\
&\leq e^{Kd(x, y)^{\theta}}\sum_{a = 0}^{\lfloor \beta - \beta_y \rfloor}e^{A(ay)}\varphi(ay) 
= e^{Kd(x, y)^{\theta}}L_A(\psi)(y) \;,  
\end{align*}
which implies that $L_A(\psi) \in \Lambda_K(\Sigma_\beta)$.

The {\bf dual Ruelle operator} associated to $A \in \mathcal{H}_\theta(\Sigma_\beta)$ is given by the map that assigns to each $\mu \in \mathcal{B}(\Sigma_\beta)$ the value $L_A^*(\mu) \in \mathcal{B}(\Sigma_\beta)$, which satisfies the equation
\begin{equation*}
\int_{\Sigma_\beta}\varphi d(L_A^*(\mu)) := \int_{\Sigma_\beta} L_A(\varphi) d\mu \;,
\end{equation*}
for each $\varphi \in \mathcal{C}(\Sigma_\beta)$.

Since the shift map $\sigma:\; \Sigma_\beta \to \Sigma_\beta$ is a local homeomorphism (see for instance \cite{MR648108}). Assuming that there is $n_0 \in \mathbb{N}$ such that all the words composed only by $0$'s in the quasi-greedy $\beta$-expansion of $1$ have length less than or equal to $n_0$, it follows that the transfer operator defined in \eqref{Ruelle-operator} satisfies the Perron-Frobenius Theorem of Ruelle on the \victor{set of functions} $\Lambda_K(\Sigma_\beta)$ for the constant $K = \mathrm{Hol}_A$. That is:

\begin{remark}
\label{Perron-Frobenius}
There are $\lambda_A > 0$, a function $\psi_A \in \Lambda_{\mathrm{Hol}_A}(\Sigma_\beta)$ and a measure $\rho_A \in \mathcal{P}(\Sigma_\beta)$ such that:
\begin{enumerate}[i)]
\item $L_A(\psi_A) = \lambda_A \psi_A$; 
\item $L_A^*(\rho_A) = \lambda_A \rho_A$; 
\item $\lim_{n \to \infty}\lambda_A^{-n}L_A^n(\varphi) = \psi_A\int_{\Sigma_\beta} \varphi d\rho_A$ for any $\varphi \in \mathcal{H}_\theta(\Sigma_\beta)$. 
\end{enumerate}

Moreover, by item $iii)$, it follows that $\lambda_A$ is simple, maximal and  isolated. In addition, choosing the main eigenfunction $\psi_A$ in a smart way, we obtain that the Gibbs state $\mu_A := \psi_A d\rho_A$ belongs to $\mathcal{P}_\sigma(\Sigma_\beta)$ and results in an equilibrium state for the potential $A$ (see for details the main theorems that appear in \cite{MR1860762}, \cite{MR1860763} and \cite{MR0466492}).

\end{remark}

\section{Involution kernel}
\label{involution-kernel-section}

In this section, we propose an appropriate definition of the involution kernel for the setting of $\beta$-shifts. Besides that, we present an explicit expression for the main eigenfunction of the Ruelle operator, whose existence is guaranteed in Remark \ref{Perron-Frobenius}. The above, with the goal to obtain natural extensions to the bilateral $\beta$-shift of both, the Gibbs states $\mu_{t A}$, $t > 1$, and the accumulation points at zero temperature of the family $(\mu_{t A})_{t>1}$.
 
Given $m \in \mathbb{N}$ and $x = (x_n)_{n \geq 1},\; y = (y_n)_{n \geq 1} \in \Sigma_{\lfloor \beta \rfloor}$, we define the sequences $\tau_{y, m}(x) := (y_m,\; ...,\; y_1, x_1, x_2,\; ...\; )$ and $\tau_{y, 0}(x) := x$. Note that $\tau_{y, m}(x)$ only takes values into the alphabet $\mathcal{A}_\beta$, which implies \victor{that it also belongs} to $\Sigma_{\lfloor \beta \rfloor}$. The so called {\bf bilateral $\beta$-shift} $\widehat{\Sigma_\beta}$ is defined as 
\begin{equation}
\label{bilateral-shift}
\widehat{\Sigma_\beta} := \{ (y, x) \in \Sigma_{\lfloor \beta \rfloor} \times \Sigma_{\lfloor \beta \rfloor} :\; \tau_{y, m}(x) \in \Sigma_\beta,\; \forall m \in \mathbb{N} \cup \{0\} \} \;,
\end{equation}
with the {\bf bilateral shift map} $\widehat{\sigma} :\; \widehat{\Sigma_\beta} \to \widehat{\Sigma_\beta}$ given by $\widehat{\sigma}(y, x) := (\tau_{x, 1}(y), \sigma(x))$ acting on it. It is easy to check that the definition above coincides with the ones that appear in \cite{MR0466492} and \cite{MR648108}. Furthermore, the set $\widehat{\Sigma_\beta}$ is closed and invariant by the action of the homeomorphism $\widehat{\sigma}$ and results in a compact metric space when it is equipped with the metric
\begin{equation*}
\widehat{d}((y, x), (y', x')) := 2^{-\min\{n \in\mathbb{N} :\; y_n \neq y'_n \text{ and } x_n \neq x'_n \} + 1} \;. 
\end{equation*}

On the other hand, observe that for any sequence $\tau_{y, m}(x) \in \Sigma_\beta$, with $m \in \mathbb{N}$, \victor{the sequence} $x$ also belongs to $\Sigma_\beta$. The above, because we have $x = \sigma^m(\tau_{y, m}(x))$ and $\sigma(\Sigma_\beta) \subset \Sigma_\beta$. In fact, \victor{the above} is equivalent to say that for any pair $(y, x) \in \widehat{\Sigma_\beta}$ \victor{we have} $x \in \Sigma_\beta$.  Conversely, if the sequence $x$ belongs to $\Sigma_\beta$, we obtain that $\tau_{0^{\infty}, m}(x) \in \Sigma_\beta$ for each $m \in \mathbb{N}$. The foregoing, because $\beta_{\tau_{0^{\infty}, m}(x)} = \beta^{-m}\beta_x \leq \beta_x$, which implies that the pair $(0^{\infty}, x)$ also belongs to the so called bilateral $\beta$-shift. 

By the above, it follows that the bilateral $\beta$-shift admits a Cartesian product decomposition of the form $\widehat{\Sigma_\beta} = \Sigma_\beta^\intercal \times \Sigma_\beta$, where $\Sigma_\beta^\intercal \subset \Sigma_{\lfloor \beta \rfloor}$ results in a $\sigma$-invariant closed set which we call {\bf transpose $\beta$-shift}. In the next lemma we present a characterization of the subshift $\Sigma_\beta^\intercal$ in terms of finite admissible words in a similar way that the one given for $\Sigma_\beta$ that appears in \eqref{greedy-algorithm}, which is obtained using the so called greedy algorithm. Besides that, we use that characterization to prove our claims of closeness and $\sigma$-invariance of $\Sigma_\beta^\intercal$.  

\begin{lemma}
\label{transpose-shift-lemma}
Consider $\beta > 1$ and the Cartesian product decomposition $\widehat{\Sigma_\beta} = \Sigma_\beta^\intercal \times \Sigma_\beta$ of the bilateral $\beta$-shift. Then, the so called \victor{transpose $\beta$-shift} can be characterized in the following way
\begin{equation}
\label{transpose-shift}
\Sigma_\beta^\intercal = \Bigl\{y \in \Sigma_{\lfloor \beta \rfloor} :\; y_{k+m-1}\; ...\; y_k \prec x^\beta_1\; ...\; x^\beta_m,\; \forall k, m \in \mathbb{N} \Bigr\} \;.
\end{equation}   
\end{lemma}

\begin{proof}
First note that the set that appears in the right side of \eqref{transpose-shift} is $\sigma$-invariant. Indeed, assuming that $y$ belongs to that set, it follows that for any $k, m \in \mathbb{N}$ is satisfied $y_{k+m}\; ...\; y_{k+1} \prec x^\beta_1\; ...\; x^\beta_m$. In other words, we have that $\sigma(y)$ also belongs to the set appearing in the right side of \eqref{transpose-shift}, such as we wanted to prove. 

Hence, in order to prove the $\sigma$-invariance of the so called transpose $\beta$-shift, it is enough to prove that \eqref{transpose-shift} holds.

Assume that $(y, x) \in \widehat{\Sigma_\beta}$, by \eqref{bilateral-shift}, for each $n \in \mathbb{N}$ is satisfied $\tau_{y, n}(x) \in \Sigma_\beta$. The above implies, as a consequence of \eqref{greedy-algorithm}, that given arbitrary values $k, m \in \mathbb{N}$, we have $y_{k+m-1}\; ...\; y_k \prec x^\beta_1\; ...\; x^\beta_m \;$. Hence, we obtain that
\[
\Sigma_\beta^\intercal \subseteq \Bigl\{y \in \Sigma_{\lfloor \beta \rfloor} :\; y_{k+m-1}\; ...\; y_k \prec x^\beta_1\; ...\; x^\beta_m,\; \forall k, m \in \mathbb{N} \Bigr\} \;. 
\] 

Conversely, assuming that $y \in \Sigma_{\lfloor \beta \rfloor}$ satisfies $y_{k+m-1}\; ...\; y_k \prec x^\beta_1\; ...\; x^\beta_m$ for arbitrary values $k, m \in \mathbb{N}$. In particular, we have $y_{m}\; ...\; y_1 \prec x^\beta_1\; ...\; x^\beta_m$ for each $m \in \mathbb{N}$ which is equivalent to say that $\tau_{y, m}(0^{\infty}) \in \Sigma_\beta$ for each $m \in \mathbb{N}$. Therefore, $(y, 0^{\infty}) \in \widehat{\Sigma_\beta}$ and, thus, $y \in \Sigma_\beta^\intercal$. By the above, we have
\[
\Sigma_\beta^\intercal \supseteq \Bigl\{y \in \Sigma_{\lfloor \beta \rfloor} :\; y_{k+m-1}\; ...\; y_k \prec x^\beta_1\; ...\; x^\beta_m,\; \forall k, m \in \mathbb{N} \Bigr\} \;, 
\]
which shows that \eqref{transpose-shift} holds.

In order to finish the proof, by \eqref{transpose-shift}, we have that the closeness of $\Sigma_\beta^\intercal$ follows from the greedy algorithm in the same way that in the case of $\beta$-shifts \victor{(see for instance \cite{MR1020481}, \cite{MR1078082} and \cite{MR0142719}).}   
\end{proof}

\begin{remark}
\label{sequences}
Observe that for each pair $(y, x) \in \widehat{\Sigma_\beta}$ and any $n \in \mathbb{N}$ \victor{we have} $\widehat{\sigma}^n(y, x) = (\tau_{x, n}(y), \sigma^n(x))$ also belongs to the set $\widehat{\Sigma_\beta}$. Therefore, it follows from Lemma \ref{transpose-shift-lemma} that the sequence $\tau_{x, n}(y)$ belongs to $\Sigma_\beta^\intercal$ for any pair $(y, x) \in \widehat{\Sigma_\beta}$ and each $n \in \mathbb{N}$.
\end{remark}

Given a word $w = w_1\; ...\; w_l$, we define the {\bf transpose word} of $w$ as the one given by the expression $w^\intercal = w_l\; ...\; w_1$. Observe that for any $w \in \mathcal{L}(\Sigma_\beta)$, the word $w^\intercal$ belongs to the language $\mathcal{L}(\Sigma_\beta^\intercal)$ and, conversely, if $w \in \mathcal{L}(\Sigma_\beta^\intercal)$, it follows that the transpose word $w^\intercal$ belongs to $\mathcal{L}(\Sigma_\beta)$. Furthermore, it is not difficult to check that for any pair of words $v, w$ \victor{we have} both, $(vw)^\intercal = w^\intercal v^\intercal$ and $(w^\intercal)^\intercal = w$.   

In order to find an explicit expression for the main eigenfunction $\psi_A$ associated to $A \in \mathcal{H}_\theta(\Sigma_\beta)$, whose existence is guaranteed by the Perron-Frobenius Theorem of Ruelle \victor{(see details in \cite{MR1860762}, \cite{MR1860763} and \cite{MR0466492}),} we introduce a definition of involution kernel adapted to our matter using properties of $\beta$-expansions. However, first we need to check that the subshift $\Sigma_\beta^\intercal$ satisfies suitable conditions in order to develop the theory of Perron-Frobenius.

\begin{lemma}
\label{specification-lemma}
Consider $\beta > 1$, assume that $\Sigma_\beta$ satisfies the specification property and set the bilateral $\beta$-shift $\widehat{\Sigma_\beta}$ such as appears in \eqref{bilateral-shift}. Then, the so called transpose $\beta$-shift also satisfies the specification property.  
\end{lemma}
\begin{proof}
Since $\Sigma_\beta$ has the specification property, there is $n_0 \in \mathbb{N}$ such that for each pair of words $u, w \in \mathcal{L}(\Sigma_\beta)$, there is a \victor{word $v \in \mathcal{A}_\beta^{n_0}$} such that $uvw \in \mathcal{L}(\Sigma_\beta)$. Consider a pair of words $u, w \in \mathcal{L}(\Sigma_\beta^\intercal)$, thus, we have that the corresponding transpose words $u^\intercal, w^\intercal$ belong to the language $\mathcal{L}(\Sigma_\beta)$. By the specification property, there is a \victor{word $v^\intercal \in \mathcal{A}_\beta^{n_0}$} such that $w^\intercal v^\intercal u^\intercal \in \mathcal{L}(\Sigma_\beta)$, which implies that $(w^\intercal v^\intercal u^\intercal)^\intercal = uvw \in \mathcal{L}(\Sigma_\beta^\intercal)$. That is, the subshift $\Sigma_\beta^\intercal$ also satisfies the specification property, such as we wanted to prove.
\end{proof}

Note that the above lemma implies, in particular, that the subshift $\Sigma_\beta^\intercal$ is topologically mixing under the assumption that there is $n_0 \in \mathbb{N}$ such that all the words composed only by $0$'s in the quasi-greedy $\beta$-expansion of $1$ have length less than or equal to $n_0$ \victor{(see for details \cite{MR878240}, \cite{zbMATH03514123} and \cite{MR1484730}). The foregoing is essential to prove some properties of the Ruelle operator on the subshift $\Sigma_\beta^\intercal$.} In order to do that, we need to introduce some definitions that establish a relation between the subshifts $\Sigma_\beta$ and $\Sigma_\beta^\intercal$.

We say that a map $W :\; \widehat{\Sigma_\beta} \to \mathbb{R}$ is an {\bf involution kernel} associated to the potential $A :\; \Sigma_\beta \to \mathbb{R}$, if for any pair of the form $(y, ax) \in \widehat{\Sigma_\beta}$, the potentials $\widehat{A} :\; \widehat{\Sigma_\beta} \to \mathbb{R}$ defined as $\widehat{A}(y, x) := A(x)$ and $\widehat{A^\intercal} :\; \widehat{\Sigma_\beta} \to \mathbb{R}$ given by 
\begin{equation*}
\widehat{A^\intercal}(ay, x) := \widehat{A}(y, ax) + W(y, ax) - W(ay, x) \;,
\end{equation*}
are in such a way that the potential $\widehat{A^\intercal}$ \victor{doesn't depends on} the second coordinate. \victor{Under those assumptions,} we use the notation $A^\intercal(y) := \widehat{A^\intercal}(y, x)$ and we call the potential $A^\intercal :\; \Sigma_\beta^\intercal \to \mathbb{R}$ as the {\bf transpose potential} of $A$. Some references with lots of details and examples related to the involution kernel in classical settings of the XY model, the full-shift on finite alphabets and \victor{XY models with Markovian structure appear in \cite{MR3114331}, \cite{MR2210682}, \cite{zbMATH07220150} and \cite{SV20}.}

Fixing $x' \in \Sigma_\beta$, define the map $W :\; \widehat{\Sigma_\beta} \to \mathbb{R}$ by the equation
\begin{equation}
\label{involution-kernel}
W(y, x) := \sum_{n = 1}^\infty A(\tau_{y, n}(x)) - A(\tau_{y, n}(x')) \,.
\end{equation}  

By \eqref{bilateral-shift}, the map $W$ is well defined on $\widehat{\Sigma_\beta}$ and an easy calculation shows that it is an involution kernel associated to $A$ (see for details \cite{MR2210682}). Furthermore, if $A \in \mathcal{H}_\theta(\Sigma_\beta)$, it follows that $W \in \mathcal{H}_\theta(\widehat{\Sigma_\beta})$, which also implies that the corresponding transpose potential $A^\intercal$ preserves the conditions of regularity of the potential $A$ via the involution kernel that appears in \eqref{involution-kernel}. 

Observe that the shift map $\sigma :\; \Sigma_\beta^\intercal \to \Sigma_\beta^\intercal$ is a local homeomorphism. Indeed, for each $y \in \Sigma_\beta^\intercal$ and any $m \in \mathbb{N}$ \victor{we have} $\tau_{y, m}(x)$ belongs to $\Sigma_\beta$ when $x \in \Sigma_\beta$. \victor{The above is equivalent to saying that} for any $y \in \Sigma_\beta^\intercal$ the sequence $\tau_{y, m}(x^\beta)$ belongs to $\Sigma_\beta$, which implies that the sequence $ay$ belongs to $\Sigma_\beta^\intercal$ for each $a \in \{0,\; ...\; ,\; x^\beta_1\}$. \victor{By the above, it follows that the map $\sigma :\; \Sigma_\beta^\intercal \to \Sigma_\beta^\intercal$ has $x^\beta_1 + 1$ inverse branches. In particular, the map $\sigma|_{[a]} :\; [a] \subset \Sigma_\beta^\intercal \to \Sigma_\beta^\intercal$ results in a homeomorphism for each $a \in \{0,\; ...\; ,\; x^\beta_1\}$. Therefore,} we can define the Ruelle operator on the transpose $\beta$-shift, as the map that assigns to each $\varphi^\intercal \in \mathcal{C}(\Sigma_\beta^\intercal)$ the value \victor{
\begin{equation*}
L_{A^\intercal}(\varphi^\intercal)(y) := \sum_{\sigma(z) = y}e^{A^\intercal(z)}\varphi^\intercal(z) = \sum_{a = 0}^{x^\beta_1}e^{A^\intercal(ay)}\varphi^\intercal(ay) \;. 
\end{equation*}}

Furthermore, it is not difficult to check that the transfer operator $L_{A^\intercal}$ preserves the set $\Lambda_{\mathrm{Hol}_{A^\intercal}}(\Sigma_\beta^\intercal)$ (where $\Lambda_K(\; \cdot \;)$ is defined as in \eqref{convex-set}). Hence, assuming that there exists $n_0 \in \mathbb{N}$ such that all the words composed only by $0$'s in the quasi-greedy $\beta$-expansion of $1$ have length less than or equal to $n_0$, which implies specification on $\Sigma_\beta$ (see for details \cite{MR878240}), \victor{thus, by Lemma \ref{specification-lemma},} it follows that $L_{A^\intercal}$ satisfies the Perron-Frobenius Theorem of Ruelle. That is:
\begin{remark} 
\label{Perron-Frobenius-transpose}
There are $\lambda_{A^\intercal} > 0$, a function $\psi_{A^\intercal} \in \Lambda_{\mathrm{Hol}_{A^\intercal}}(\Sigma_\beta^\intercal)$ and a measure $\rho_{A^\intercal} \in \mathcal{P}(\Sigma_\beta^\intercal)$ such that:
\begin{enumerate}[i)] 
\item $L_{A^\intercal}(\psi_{A^\intercal}) = \lambda_{A^\intercal} \psi_{A^\intercal}$;
\item $L_{A^\intercal}^*(\rho_{A^\intercal}) = \lambda_{A^\intercal} \rho_{A^\intercal}$;
\item $\lim_{n \to \infty}\lambda_{A^\intercal}^{-n}L_{A^\intercal}^n(\varphi^\intercal) = \psi_{A^\intercal}\int_{\Sigma_\beta^\intercal} \varphi^\intercal d\rho_{A^\intercal}$ for any $\varphi^\intercal \in \mathcal{H}_\theta(\Sigma_\beta^\intercal)$.
\end{enumerate} 

In this case, also as a consequence of item $iii)$, it follows by \victor{a straightforward argument} that the eigenvalue $\lambda_{A^\intercal}$ is simple, maximal and isolated. Moreover, choosing a suitable main eigenfunction $\psi_{A^\intercal}$, it follows that the Gibbs state $\mu_{A^\intercal} := \psi_{A^\intercal} d\rho_{A^\intercal}$ belongs to $\mathcal{P}_\sigma(\Sigma_\beta^\intercal)$ and results in an equilibrium state for the transpose potential $A^\intercal$ (see for details the main theorems in \cite{MR1860762} and \cite{MR1860763}).
\end{remark}

In the following lemma we present an interesting relation between the Ruelle operator associated to a potential $A \in \mathcal{H}_\theta(\Sigma_\beta)$ and the Ruelle operator associated to its corresponding transpose $A^\intercal \in \mathcal{H}_\theta(\Sigma_\beta^\intercal)$, which will be useful at the moment to find an explicit expression for the main eigenfunction of the operator $L_A$. 

\begin{lemma}
\label{lemma-duality}
Consider $\beta > 1$ and $A \in \mathcal{H}_\theta(\Sigma_\beta)$. Let $W \in \mathcal{H}_\theta(\widehat{\Sigma_\beta})$ be an involution kernel associated to $A$ and let $A^\intercal \in \mathcal{H}_\theta(\Sigma_\beta^\intercal)$ be its corresponding transpose potential. Then, for any $(y, x) \in \widehat{\Sigma_\beta}$, \victor{we obtain that}
\[
L_{A^\intercal}({\bf 1}_{\widehat{\Sigma_\beta}}(\; \cdot, x)e^{W(\; \cdot, x)})(y) = L_A({\bf 1}_{\widehat{\Sigma_\beta}}(y,\; \cdot)e^{W(y,\; \cdot)})(x) \,.
\]  
\end{lemma}
\begin{proof} 
First note that the maps $e^{(\widehat{A^\intercal} + W)}$ and $e^{(\widehat{A} + W)}$ can be extended to continuous functions and, thus, bounded functions defined on $\Sigma_{\lfloor \beta \rfloor} \times \Sigma_{\lfloor \beta \rfloor}$. Therefore, we can guarantee that the maps ${\bf 1}_{\widehat{\Sigma_\beta}}e^{(\widehat{A^\intercal} + W)}$ and ${\bf 1}_{\widehat{\Sigma_\beta}}e^{(\widehat{A} + W)}$ are well defined and are equal to $0$ for any pair $(y, x) \in (\Sigma_{\lfloor \beta \rfloor} \times \Sigma_{\lfloor \beta \rfloor}) \setminus \widehat{\Sigma_\beta}$. 

Besides that, since $\widehat{\Sigma_\beta}$ is a $\widehat{\sigma}$-invariant set and $\widehat{\sigma}$ is a homeomorphism, we obtain that for any pair of the form $(y, ax) \in \Sigma_{\lfloor \beta \rfloor} \times \Sigma_{\lfloor \beta \rfloor}$ \victor{we have} 
\[
{\bf 1}_{\widehat{\Sigma_\beta}}(ay, x) = {\bf 1}_{\widehat{\Sigma_\beta}}(y, ax) \;.
\] 

Hence, fixing a pair $(y, x) \in \widehat{\Sigma_\beta}$ we obtain that \victor{
\begin{align*}
L_{A^\intercal}({\bf 1}_{\widehat{\Sigma_\beta}}(\; \cdot, x)e^{W(\; \cdot, x)})(y) 
&= \sum_{a = 0}^{x^\beta_1} {\bf 1}_{\widehat{\Sigma_\beta}}(ay, x) e^{(\widehat{A^\intercal} + W)(ay, x)} \\
=& \sum_{a = 0}^{x^\beta_1} {\bf 1}_{\widehat{\Sigma_\beta}}(y, ax) e^{(\widehat{A} + W)(y, ax)} \\
=& \sum_{a = 0}^{\lfloor \beta - \beta_x \rfloor} {\bf 1}_{\widehat{\Sigma_\beta}}(y, ax) e^{(\widehat{A} + W)(y, ax)}
= L_A({\bf 1}_{\widehat{\Sigma_\beta}}(y,\; \cdot) e^{W(y,\; \cdot)})(x) \,,
\end{align*}}
where the second last equality follows from the fact that $x \in \Sigma_\beta$. The above, because ${\bf 1}_{\widehat{\Sigma_\beta}}(y, ax) \neq 0$ if, and only if, $a \in \{0,\; ...\; ,\; \lfloor \beta - \beta_x \rfloor\}$, which is a consequence of the Cartesian product decomposition $\widehat{\Sigma_\beta} = \Sigma_\beta^\intercal \times \Sigma_\beta$ \victor{proved above and the fact $\lfloor \beta - \beta_x \rfloor \leq x^\beta_1$.}
\end{proof}

In the following lemma we present an explicit expression for the main eigenfunctions associated to the operators $L_A$ and $L_{A^\intercal}$ defined above. However, before that, it is important to point out that the following result uses strongly the decomposition of the bilateral $\beta$-shift as Cartesian product of the subshifts $\Sigma_\beta^\intercal$ and $\Sigma_\beta$, joint with the fact that specification in the $\beta$-shift implies specification in the so called transpose $\beta$-shift \victor{(which was proved in Lemma \ref{specification-lemma}).}

\begin{lemma}
\label{lemma-eigenfunction}
Consider $\beta > 1$ and \victor{assume that $\Sigma_\beta$ satisfies the specification property}. Let $W \in \mathcal{H}_\theta(\widehat{\Sigma_\beta})$ be an involution kernel associated to $A \in \mathcal{H}_\theta(\Sigma_\beta)$ with $A^\intercal \in \mathcal{H}_\theta(\Sigma_\beta^\intercal)$ its corresponding transpose potential. Set $\psi :\; \Sigma_\beta \to \mathbb{R}$ and $\psi^\intercal :\; \Sigma_\beta^\intercal \to \mathbb{R}$ as 
\begin{equation*}
\psi := \int_{\Sigma_\beta^\intercal} {\bf 1}_{\widehat{\Sigma_\beta}}(y, \cdot) e^{W (y, \cdot)} d\rho_{A^\intercal}(y) \;,
\end{equation*}
and 
\begin{equation*}
\psi^\intercal := \int_{\Sigma_\beta} {\bf 1}_{\widehat{\Sigma_\beta}}(\cdot, x) e^{W (\cdot, x)} d\rho_{A}(x) \;.
\end{equation*} 

Then, $\psi$ belongs to $\Lambda_{\mathrm{Hol}_A} (\Sigma_\beta)$ and results in a main eigenfunction of the operator $L_A$. Similarly, $\psi^\intercal$ belongs to $\Lambda_{\mathrm{Hol}_{A^\intercal}}(\Sigma_\beta^\intercal)$ and results in a main eigenfunction of the operator $L_{A^\intercal}$. 
\end{lemma}
\begin{proof}
By the Cartesian product decomposition $\widehat{\Sigma_\beta} = \Sigma_\beta^\intercal \times \Sigma_\beta$ of the bilateral $\beta$-shift and Lemma \ref{lemma-duality}, we obtain that for each $x \in \Sigma_\beta$ is satisfied
\begin{align*}
\psi(x) 
&= \int_{\Sigma_\beta^\intercal} {\bf 1}_{\widehat{\Sigma_\beta}}(y, x)e^{W (y, x)}d\rho_{A^\intercal}(y) \\
&= \frac{1}{\lambda_{A^\intercal}}\int_{\Sigma_\beta^\intercal} {\bf 1}_{\widehat{\Sigma_\beta}}(y, x) e^{W(y, x)}d(L^*_{A^\intercal}\rho_{A^\intercal})(y) \\
&= \frac{1}{\lambda_{A^\intercal}}\int_{\Sigma_\beta^\intercal}L_{A^\intercal}({\bf 1}_{\widehat{\Sigma_\beta}}(\; \cdot, x)e^{W(\; \cdot, x)})(y)d\rho_{A^\intercal}(y) \\
&= \frac{1}{\lambda_{A^\intercal}}\int_{\Sigma_\beta^\intercal}L_A({\bf 1}_{\widehat{\Sigma_\beta}}(y,\; \cdot)e^{W(y,\; \cdot)})(x)d\rho_{A^\intercal}(y) \\
&= \frac{1}{\lambda_{A^\intercal}}L_A\Bigl(\int_{\Sigma_\beta^\intercal}\,{\bf 1}_{\widehat{\Sigma_\beta}}(y,\; \cdot) e^{\,W_A(y,\; \cdot)}d\rho_{A^\intercal}(y)\Bigr)(x) 
= \frac{1}{\lambda_{A^\intercal}}L_A(\psi)(x) \,.
\end{align*}

Thus, in order to conclude the proof of the first part of this lemma, we only need to show that $\lambda_{A^\intercal} = \lambda_A$ and prove that $\psi \in \Lambda_{\mathrm{Hol}_A} (\Sigma_\beta)$. The first claim is a consequence of the existence of the limit in the item $iii)$ of Remark \ref{Perron-Frobenius} and the fact that $\psi > 0$. The above, because that limit implies that the unique strictly positive eigenfunctions of the operator $L_A$ are the ones associated to the main eigenvalue $\lambda_A$ (see for details chapter $2$ in \cite{MR1085356}). In addition, since the eigenvalue $\lambda_A$ is also a simple eigenvalue, it follows that $\psi = \kappa\psi_A$ and, thus, $\psi \in \Lambda_{\mathrm{Hol}_A} (\Sigma_\beta)$, which proves our second claim.

The proof of the second part of this lemma, i.e., that the function $\psi^\intercal$ is a main eigenfunction of the operator $L_{A^\intercal}$ belonging to the set $\Lambda_{\mathrm{Hol}_{A^\intercal}} (\Sigma_\beta^\intercal)$ follows in the \victor{same way as} the one that appears above, but as a consequence of Remark \ref{Perron-Frobenius-transpose}.
\end{proof}

Note that the eigenfunction $\psi$ obtained in Lemma \ref{lemma-eigenfunction} is not necessarily normalized in the sense that $\int_{\Sigma_\beta} \psi d\rho_A = 1$, which is essential at the moment to define the Gibbs state associated to the potential $A$. However, defining the constant
\begin{equation}
\label{normalized-constant}
c_A := \log\left(\int_{\widehat{\Sigma_\beta}} e^W d(\rho_{A^\intercal} \times \rho_A)\right) \;,
\end{equation}
we obtain that 
\begin{equation}
\label{normalized-eigenfunction}
\int_{\Sigma_\beta^\intercal} e^{W(y,\; \cdot) - c_A}d\rho_{A^\intercal}(y) = \psi_A \;,
\end{equation} 
where $\psi_A$ satisfies the properties that appear in Remark \ref{Perron-Frobenius}, among them, the desired normalized property in the sense that $\int_{\Sigma_\beta} \psi_A d\rho_A = 1$. 

\victor{The next proposition} presents a natural extension of both, the Gibbs state $\mu_A \in \mathcal{P}_\sigma(\Sigma_\beta)$ that appears in Remark \ref{Perron-Frobenius} and the Gibbs state $\mu_{A^\intercal} \in \mathcal{P}_\sigma(\Sigma_\beta^\intercal)$ that appears in Remark \ref{Perron-Frobenius-transpose}, to a probability measure $\mu_{\widehat{A}}$ defined on the \victor{Borel sets} of the so called bilateral $\beta$-shift which, in particular, belongs to the transport plan $\Gamma(\mu_{A^\intercal}, \mu_A)$. The above is an interesting property of this couple of Gibbs states that \victor{will be useful to prove} the large deviations principle that appears in section \ref{large-deviations-section}. 

\begin{proposition}
\label{natural-extension-proposition}
Consider $\beta > 1$ and \victor{assume that $\Sigma_\beta$ satisfies the specification property.} Let $W \in \mathcal{H}_\theta(\widehat{\Sigma_\beta})$ be an involution kernel associated to $A \in \mathcal{H}_\theta(\Sigma_\beta)$ with $A^\intercal \in \mathcal{H}_\theta(\Sigma_\beta^\intercal)$ its corresponding transpose potential and $c_A$ a constant such as appears in \eqref{normalized-constant}. Define the measure $\mu_{\widehat{A}} \in \mathcal{P}(\widehat{\Sigma_\beta})$ as
\begin{equation*}
\mu_{\widehat{A}} := e^{W - c_A}d(\rho_{A^\intercal} \times \rho_A) \;.
\end{equation*} 

Then, $\mu_{\widehat{A}}$ belongs to the transport plan $\Gamma(\mu_{A^\intercal}, \mu_{A})$ and the family $(\mu_{t\widehat{A}})_{t > 1}$ has an accumulation point $\widehat{\mu}_\infty$ at $\infty$ belonging to the transport plan $\Gamma(\mu_\infty^\intercal, \mu_\infty)$, where $\mu_\infty$ is an accumulation point of the family $(\mu_{tA})_{t > 1}$ at $\infty$ and $\mu_\infty^\intercal$ is an accumulation point of the family $(\mu_{tA^\intercal})_{t > 1}$ at $\infty$. Furthermore, we have
\begin{equation}
\label{bilateral-maximizing-value}
\int_{\widehat{\Sigma_\beta}}\widehat{A} d\widehat{\mu}_\infty = m(A) \;,
\end{equation} 
where $m(A)$ is the maximizing value that appears defined in \eqref{maximizing-value}.
\end{proposition}
\begin{proof}
Consider $\varphi \in \mathcal{C}(\Sigma_\beta)$ and $\psi_A$ such as appears in \eqref{normalized-eigenfunction}. Defining $\widehat{\varphi} \in \mathcal{C}(\widehat{\Sigma_\beta})$ as $\widehat{\varphi}(y, x) = \varphi(x)$ for each pair $(y, x) \in \widehat{\Sigma_\beta}$, we obtain that
\begin{align*}
\int_{\widehat{\Sigma_\beta}} \widehat{\varphi} d\mu_{\widehat{A}}
&= \int_{\widehat{\Sigma_\beta}} \widehat{\varphi} e^{W - c_A}d(\rho_{A^\intercal} \times \rho_A) \\
&= \int_{\Sigma_\beta} \varphi(x) \int_{\Sigma_\beta^\intercal}  {\bf 1}_{\widehat{\Sigma_\beta}}(y, x) e^{W(y, x) - c_A} d\rho_{A^\intercal}(y) d\rho_A(x) \\
&= \int_{\Sigma_\beta} \varphi \psi_A d\rho_A
= \int_{\Sigma_\beta} \varphi d\mu_A \;.
\end{align*}

In a similar way, it is not difficult to prove that for any $\varphi^\intercal \in \mathcal{C}(\Sigma_\beta^\intercal)$ is satisfied 
\[
\int_{\widehat{\Sigma_\beta}} \widehat{\varphi^\intercal} d\mu_{\widehat{A}} = \int_{\Sigma_\beta^\intercal} \varphi^\intercal d\mu_{A^\intercal} \;,
\]
where the function $\widehat{\varphi^\intercal} \in \mathcal{C}(\widehat{\Sigma_\beta^\intercal})$ is given by the equation $\widehat{\varphi^\intercal}(y, x) = \varphi^\intercal(y)$ for each $(y, x) \in \widehat{\Sigma_\beta}$.

By the above, it follows that for any pair of functions $\varphi^\intercal \in \mathcal{C}(\Sigma_\beta^\intercal)$ and $\varphi \in \mathcal{C}(\Sigma_\beta)$ is satisfied 
\[
\int_{\widehat{\Sigma_\beta}} (\varphi^\intercal(y) + \varphi(x)) d\mu_{\widehat{A}}(y, x) = \int_{\Sigma_\beta^\intercal} \varphi^\intercal d\mu_{A^\intercal} + \int_{\Sigma_\beta} \varphi d\mu_A \;,
\]
which implies that the probability measure $\mu_{\widehat{A}}$ also belongs to the transport plan $\Gamma(\mu_{A^\intercal}, \mu_{A})$ (see for details \cite{MR1964483}).

On the other hand, by compactness of $\widehat{\Sigma_\beta}$, we have existence of accumulation points of the family $(\mu_{t\widehat{A}})_{t>1}$ at $\infty$. That is, there is a strictly increasing sequence $(t_n)_{n \geq 1}$ taking values into the interval $(1, \infty)$, such that
\begin{equation}
\label{limit-weak*-bilateral}
\lim_{n \to \infty} \mu_{t_n\widehat{A}} = \widehat{\mu}_\infty, 
\end{equation}
where the limit above is taken in the weak* topology. Furthermore, taking a subsequence if necessary, by compactness of $\Sigma_\beta$ and $\Sigma_\beta^\intercal$, it follows that
\begin{equation}
\label{limit-weak*}
\lim_{n \to \infty} \mu_{t_n A} = \mu_\infty \;\text{ and }\; \lim_{n \to \infty} \mu_{t_n A^\intercal} = \mu_\infty^\intercal \;,
\end{equation}
in the weak* topology.

In particular, by \eqref{limit-weak*-bilateral} and \eqref{limit-weak*}, for each $\varphi \in \mathcal{C}(\Sigma_\beta)$ is satisfied
\begin{equation}
\label{bilateral-maximizing}
\int_{\widehat{\Sigma_\beta}} \widehat{\varphi} d\widehat{\mu}_\infty 
= \lim_{n \to \infty} \int_{\widehat{\Sigma_\beta}} \widehat{\varphi} d\mu_{t_n\widehat{A}} 
= \lim_{n \to \infty} \int_{\Sigma_\beta} \varphi d\mu_{t_n A}
= \int_{\Sigma_\beta} \varphi d\mu_\infty \;,
\end{equation}
and for any $\varphi^\intercal \in \mathcal{C}(\Sigma_\beta^\intercal)$, we have
\[
\int_{\widehat{\Sigma_\beta}} \widehat{\varphi^\intercal} d\widehat{\mu}_\infty 
= \lim_{n \to \infty} \int_{\widehat{\Sigma_\beta}} \widehat{\varphi^\intercal} d\mu_{t_n\widehat{A}} 
= \lim_{n \to \infty} \int_{\Sigma_\beta^\intercal} \varphi^\intercal d\mu_{t_n A^\intercal}
= \int_{\Sigma_\beta^\intercal} \varphi^\intercal d\mu_\infty^\intercal \;.
\]

By the above, we obtain that
\begin{equation*}
\int_{\widehat{\Sigma_\beta}} (\varphi^\intercal(y) + \varphi(x)) d\widehat{\mu}_\infty(y, x) = \int_{\Sigma_\beta^\intercal} \varphi^\intercal d\mu_\infty^\intercal + \int_{\Sigma_\beta} \varphi d\mu_\infty \;,
\end{equation*}
which implies that $\widehat{\mu}_\infty \in \Gamma(\mu_\infty^\intercal, \mu_\infty)$ (see \cite{MR1964483}). Furthermore, replacing by $\varphi = A$ in \eqref{bilateral-maximizing}, we obtain \eqref{bilateral-maximizing-value}, which concludes the proof of this lemma.
\end{proof}

It is widely known that \victor{one of the main interests} of the area of Ergodic Optimization is to identify the support of the maximizing measures associated to a fixed potential. \victor{In order to do that,} the so called calibrated sub-actions arise as a useful tool to identify such supports. Below we introduce a definition of \victor{those observables} associated to the potential $A$. 

\victor{First observe} that given a subshift $Y \subset \Sigma_{\lfloor \beta \rfloor}$, for any $\varphi \in \mathcal{H}_\theta(Y)$ and each $t > 1$ is satisfied $\mathrm{Hol}_{t \varphi} \leq t\mathrm{Hol}_{\varphi}$. By the above, we obtain that the function $\psi_{tA}$ belongs to $\Lambda_{t\mathrm{Hol}_A}(\Sigma_\beta)$ and the function $\psi_{tA^\intercal}$ belongs to $\Lambda_{t\mathrm{Hol_{A^\intercal}}}(\Sigma_\beta^\intercal)$ when $t > 1$. Therefore, the families $(\frac{1}{t}\log(\psi_{tA}))_{t > 1}$ and $(\frac{1}{t}\log(\psi_{tA^\intercal}))_{t > 1}$ are equicontinuous and uniformly bounded, which implies, as a consequence of the Arzela-Ascoli's Theorem, \victor{the existence} of a strictly increasing sequence $(t_n)_{n \geq 1}$ (which is actually taken as a subsequence of the one satisfying \eqref{limit-weak*-bilateral} and \eqref{limit-weak*}), such that the following uniform limits \victor{exist}:
\begin{equation}
\label{calibrated-sub-actions}
V := \lim_{n \to \infty}\frac{1}{t_n}\log(\psi_{t_n A}) \ \ \text{ and } \ \ V^\intercal := \lim_{n \to \infty}\frac{1}{t_n}\log(\psi_{t_n A^\intercal}) \;.
\end{equation}

Furthermore, it follows that $V \in \mathcal{H}_\theta(\Sigma_\beta)$ and $V^\intercal \in \mathcal{H}_\theta(\Sigma_\beta^\intercal)$ (see for instance \cite{MR2864625}). Besides that, since the cylinders on $\Sigma_\beta$ and $\Sigma_\beta^\intercal$ have empty boundary with respect to the product topology, it follows that the metric entropy map $\mu \mapsto h(\mu)$ is upper semi-continuous on both, the set $\mathcal{P}_\sigma(\Sigma_\beta)$ and the set $\mathcal{P}_\sigma(\Sigma_\beta^\intercal)$. Therefore, for any accumulation point $\mu_\infty$ of the family $(\mu_{tA})_{t > 1}$ when $t$ goes to $\infty$, \victor{we have}
\[
\lim_{t \to \infty} h(\mu_{tA}) = h(\mu_\infty) = \sup\{h(\mu) :\; \mu \in \mathcal{P}_{\max}(A)\} \;.
\]

The above, joint with the variational principle (see for instance \cite{MR3114331}), implies that
\begin{equation}
\label{mean-eigenvalues}
\lim_{t \to \infty} \frac{1}{t}\log(\lambda_{tA}) = m(A)\;.
\end{equation}

Following the same argument \victor{as above,} but applied to the family $(\mu_{t A^\intercal})_{t > 1}$, where $A^\intercal$ is the transpose potential, it follows that \victor{
\[
\lim_{t \to \infty} \frac{1}{t}\log(\lambda_{t A^\intercal}) = m(A^\intercal) \;.
\]}

In particular, since $\lambda_{t A} = \lambda_{t A^\intercal}$, by \eqref{bilateral-maximizing-value}, we are able to define the {\bf general maximizing value} given by 
\begin{equation}
\label{general-maximizing-value}
m := m(A) = m(A^\intercal) = \int_{\widehat{\Sigma_\beta}} \widehat{A} d\widehat{\mu}_\infty \;.
\end{equation} 

We claim that both of the maps appearing in \eqref{calibrated-sub-actions} result in calibrated sub-actions associated to the potentials $A$ and $A^\intercal$ respectively. Indeed, for each $n \in \mathbb{N}$ is satisfied
\begin{align*}
\frac{1}{t_n}\log(\lambda_{t_n A}) 
&= \frac{1}{t_n}\log(L_{t_n A}(\psi_{t_n A})(x)) - \frac{1}{t_n}\log(\psi_{t_n A})(x) \\
&= \frac{1}{t_n}\log\Bigl(\sum_{\sigma(z) = x}e^{t_n(A(z) + \frac{1}{t_n}\log(\psi_{t_n A})(z))}\Bigr) - \frac{1}{t_n}\log(\psi_{t_n A})(x) \;.
\end{align*}

Thus, taking the limit when $n \to \infty$ (see Lemma $4$ in \cite{MR3377291} for details), by \eqref{mean-eigenvalues}, we obtain that for each $x \in \Sigma_\beta$ is satisfied
\begin{align*}
m 
&= \sup\{A(z) + V(z) :\; \sigma(z) = x\} - V(x) \\
&= \sup\{A(z) + V(z) - V(x) :\; \sigma(z) = x \} \;.
\end{align*}

Thus, by \eqref{calibrated-sub-action}, the map $V \in \mathcal{H}_\theta(\Sigma_\beta)$ is a calibrated sub-action associated to the potential $A$. Moreover, following a similar procedure to the one that appears above, we are able to show that $V^\intercal$ is a calibrated sub-action associated to $A^\intercal$, \victor{as we claim above.} That is, for each $y \in \Sigma_\beta^\intercal$ is satisfied
\[
m = \sup\{ A^\intercal(z) + V^\intercal(z) - V^\intercal(y) :\; \sigma(z) = y\} \;.
\]

Hereafter, we will use the following notation \victor{
\begin{equation}
\label{sup}
\gamma := \sup_{(y, x) \in \widehat{\Sigma_\beta}}\{W(y, x) - V(x) - V^\intercal(y)\} \;,
\end{equation}}
where $V$ and $V^\intercal$ are the calibrated sub-actions obtained in \eqref{calibrated-sub-actions}. 

The following lemma gives some important tools to prove the main theorem of this paper. \victor{More specifically, it provides} a characterization of the value $\gamma$ defined in \eqref{sup} in terms of limits \victor{depending on} the strictly increasing sequence $(t_n)_{n \geq 1}$ satisfying \eqref{calibrated-sub-actions}. The statement of the result is the following.

\begin{lemma}
\label{mean-integral-lemma}
Consider $\beta > 1$ and \victor{assume that $\Sigma_\beta$ satisfies the specification property.} Set \victor{$A \in \mathcal{H}_\theta(\Sigma_\beta)$} and a strictly increasing sequence $(t_n)_{n \geq 1}$ satisfying \eqref{calibrated-sub-actions}. Then, we have 
\begin{equation}
\label{mean-integral}
\lim_{n \to \infty}\frac{1}{t_n}c_{t_n A} = \gamma \;.
\end{equation}

Furthermore, the $\sup$ in $\gamma$ is attained in the points belonging to $\mathrm{supp}(\widehat{\mu}_\infty)$.
\end{lemma}
\begin{proof}
Note that for each $n \in \mathbb{N}$ is satisfied \victor{
\begin{align*}
\frac{1}{t_n} c_{t_n A}
&= \frac{1}{t_n} \log\Bigl( \int_{\widehat{\Sigma_\beta}} e^{t_n W} d(\rho_{t_n A^\intercal} \times \rho_{t_n A}) \Bigr) \\
&= \frac{1}{t_n} \log\Bigl( \int_{\widehat{\Sigma_\beta}} \frac{e^{t_n W}}{\psi_{t_n A} \psi_{t_n A^\intercal}} d(\mu_{t_n A^\intercal} \times \mu_{t_n A}) \Bigr) \\
&= \frac{1}{t_n} \log\Bigl( \int_{\widehat{\Sigma_\beta}} e^{t_n (W - \frac{1}{t_n}\log(\psi_{t_n A}) - \frac{1}{t_n}\log(\psi_{t_n A^\intercal}))} d(\mu_{t_n A^\intercal} \times \mu_{t_n A}) \Bigr) \;.
\end{align*}}

Therefore, taking the limit when $n \to \infty$, by the uniform convergence of the functions in \eqref{calibrated-sub-actions}, we obtain \eqref{mean-integral} (see for details Lemma $4$ in \cite{MR3377291}).

The last claim of the lemma follows from the properties of the weak* topology. Indeed, for any $(\widetilde{y}, \widetilde{x}) \in \mathrm{supp}(\widehat{\mu}_\infty)$ and any pair of cylinders $[u]$ and $[v]$ \victor{containing the points $\widetilde{y}$ and $\widetilde{x}$ respectively,} that is, such that $u_i = \widetilde{y}_i$ for each $1 \leq i \leq |u|$ and $v_j = \widetilde{x}_j$ for each $1 \leq j \leq |v|$, we have \victor{
\begin{align}
e^{c_{t_n A}}\mu_{t_n \widehat{A}}([u] \times [v])
&= \int_{\widehat{\Sigma_\beta}} {\bf 1}_{[u] \times [v]} \frac{e^{t_n W}}{\psi_{t_n A^\intercal} \psi_{t_n A}} d(\mu_{t_n A^\intercal} \times \mu_{t_n A}) \nonumber \\
\leq& \mu_{t_n A^\intercal}([u])\mu_{t_n A}([v])\sup\Bigl\{ \frac{e^{t_n W(y, x)}}{\psi_{t_n A^\intercal}(y) \psi_{t_n A}(x)} :\; (y, x) \in [u] \times [v] \Bigr\} \nonumber \;.
\end{align} }

Therefore, since $(\widetilde{y}, \widetilde{x}) \in \mathrm{supp}(\widehat{\mu}_\infty)$, \victor{there exist a constant} $\delta > 0$ and $N \in \mathbb{N}$, such that, \victor{for any $n \geq N$ we have} $\mu_{t_n A^\intercal}([u]) > \delta$, $\mu_{t_n A}([v]) > \delta$ and $\mu_{t_n \widehat{A}}([u] \times [v]) > \delta$. Then, by \eqref{calibrated-sub-actions}, it follows that 
\begin{equation*}
\lim_{n \to \infty}\frac{1}{t_n} c_{t_n A} \leq \sup\Bigl\{ W(y, x) - V^\intercal(y) - V(x) :\; y \in [u],\; x \in [v] \Bigr\} \;.
\end{equation*}

Furthermore, since the cylinders $[u]$ and $[v]$ are arbitrary, we obtain that 
\[
\gamma = \lim_{n \to \infty}\frac{1}{t_n} c_{t_n A} \leq W(\widetilde{y}, \widetilde{x}) - V^\intercal(\widetilde{y}) - V(\widetilde{x}) \;,
\]
which concludes the proof of this lemma.
\end{proof}

\section{Large deviations principle}
\label{large-deviations-section}

Our main goal in this section is to prove a first level large deviations principle for the matter of $\beta$-shifts. More precisely, we assume suitable conditions on the greedy $\beta$-expansion of $1$ in order to guarantee that the theory of Perron-Frobenius holds. Besides that, we use properties of the so called involution kernel, which is constructed in a similar way to the ones appearing in \cite{MR3114331} and \cite{MR2210682} (see section \ref{involution-kernel-section} for details), in order to prove a sort of large deviations principle for a function defined on the bilateral $\beta$-shift and extend that result to the \victor{natural extension to the bilateral $\beta$-shift of the rate function defined in \eqref{upper-sc-function}.} 

The so called {\bf rate function} associated to the potential \victor{$A \in \mathcal{H}_\theta(\Sigma_\beta)$} is given by the equation
\begin{equation}
\label{upper-sc-function}
I(x) := \lim_{k \to \infty} -\sum_{j = 0}^{k - 1} (V \circ \sigma - V - A + m)(\sigma^j(x)) \;,
\end{equation}
where the function \victor{$V \in \mathcal{H}_\theta(\Sigma_\beta)$} is a calibrated sub-action satisfying the same conditions \victor{as the ones appearing} in \eqref{calibrated-sub-actions} and the value $m$ is the general maximizing value that appears defined in \eqref{general-maximizing-value}. Actually, it is not difficult to check that the rate function $I$ takes values into the interval $[-\infty, 0]$, belongs to the set \victor{$\mathcal{H}_\theta(\Sigma_\beta)$} and attains its maximum value on the union of the supports of all the maximizing measures associated to the potential $A$. \victor{Indeed, for any $\mu_{\max} \in \mathcal{P}_{\max}(A)$ we have $\int_{\Sigma_\beta} I d\mu_{\max} = 0$ which implies that for each $x \in \mathrm{supp}(\mu_{\max})$ we have $I(x) = 0$.}

In order to prove the so called large deviations principle, we define a family of maps \victor{$F_{t_n, k} :\; \widehat{\Sigma_\beta} \to \mathbb{R}$,} with $(t_n)_{n \geq 1}$ a strictly increasing sequence satisfying \eqref{calibrated-sub-actions} and $k \in \mathbb{N}$. The maps belonging to the family are given by the equation
\begin{align*}
F_{t_n, k}(y, x) 
:=& -\frac{1}{t_n} c_{t_n A} + W(y, x) - \frac{1}{t_n}\log{\psi_{t_n A}}(x) - \frac{1}{t_n}\log{\psi_{t_n A^\intercal}}(y) \\
&- \sum_{j = 0}^{k - 1}\Bigl(\frac{1}{t_n}\log{\psi_{t_n A^\intercal}} \circ \sigma - \frac{1}{t_n}\log{\psi_{t_n A^\intercal}} - A^\intercal + \frac{1}{t_n}\log{\lambda_{t_n A^\intercal}}\Bigr)(\tau_{x, j + 1}(y)) \;.
\end{align*}

By Remark \ref{sequences}, each one of the maps $F_{t_n, k}$ is well defined for any pair \victor{$(y, x) \in \widehat{\Sigma_\beta}$.} Moreover, from Lemma \ref{mean-integral-lemma} and the equations \eqref{calibrated-sub-actions} and \eqref{mean-eigenvalues}, it follows that the sequence of functions $(F_{t_n, k})_{n \geq 1}$ converges uniformly when $n$ goes to $\infty$ to a map \victor{$F_k :\; \widehat{\Sigma_\beta} \to \mathbb{R}$,} given by the equation
\begin{equation}
\label{F_k-function}
F_k(y, x) = -\gamma + W(y, x) - V(x) - V^\intercal(y) - \sum_{j = 0}^{k - 1}(V^\intercal \circ \sigma - V^\intercal - A^\intercal + m)(\tau_{x, j + 1}(y)) \;.
\end{equation}

In the following lemma we prove that each one of the functions $F_k$, with $k \in \mathbb{N}$, satisfies a sort of large deviations principle, which is the main tool to prove the main theorem of this paper. The statement of the result is the following.

\begin{lemma}
\label{large-deviations-lemma-F_k}
Consider $\beta > 1$ and \victor{assume that $\Sigma_\beta$ satisfies the specification property.} Fix a strictly increasing sequence $(t_n)_{n \geq 1}$ satisfying \eqref{calibrated-sub-actions}, assume that \victor{$A \in \mathcal{H}_\theta(\Sigma_\beta)$ and let $(F_k)_{k \geq 1}$ be a sequence of functions, with $F_k$ given by the equation \eqref{F_k-function}. Then, for each $k \in \mathbb{N}$ and any $w \in \mathcal{L}(\Sigma_\beta)$ is satisfied
\begin{equation*}
\lim_{n \to \infty}\frac{1}{t_n}\log(\mu_{t_n A}([w])) = \sup\{F_k(y, x) :\; (y, x) \in \Sigma_\beta^\intercal \times [w]\} \;.
\end{equation*}} 
\end{lemma}
\begin{proof}
\victor{Consider $k \in \mathbb{N}$, \victor{$w \in \mathcal{L}(\Sigma_\beta)$} and $\epsilon > 0$, since the limits 
\[
\mu_\infty = \lim_{n \to \infty} \mu_{t_n A} \text{ and } \mu_\infty^\intercal = \lim_{n \to \infty} \mu_{t_n A^\intercal} \;,
\] 
are satisfied in the weak* topology and the sequence of functions $(F_{t_n, k})_{n \geq 1}$ converges uniformly to the map $F_k$.} By Proposition \ref{natural-extension-proposition}, there is $N_1 \in \mathbb{N}$ such that for each $n \geq N_1$ is satisfied \victor{
\begin{align}
\mu_{t_n A}([w]) 
&= \mu_{t_n \widehat{A}}(\Sigma_\beta^\intercal \times [w]) \nonumber \\
&= \int_{\widehat{\Sigma_\beta}} {\bf 1}_{\Sigma_\beta^\intercal \times [w]} e^{t_n W - c_{t_n A}} d(\rho_{t_n A^\intercal} \times \rho_{t_n A}) \nonumber \\ 
&= \int_{\widehat{\Sigma_\beta}} {\bf 1}_{\Sigma_\beta^\intercal \times [w]} e^{-c_{t_n A} + t_n W - \log(\psi_{t_n A}) - \log(\psi_{t_n A^\intercal})} d(\mu_{t_n A^\intercal} \times \mu_{t_n A}) \nonumber \\
&\leq e^{t_n \sup\{F_{t_n, k}(y, x) :\; (y, x) \in \Sigma_\beta^\intercal \times [w]\} + \epsilon} \label{inequality-F_k} \;.
\end{align}}

Hence, taking $\log$ in both of the sides of the inequality and multiplying by $\frac{1}{t_n}$, we obtain that \victor{
\begin{equation*}
\frac{1}{t_n}\log(\mu_{t_n A}([w])) 
\leq \sup\{F_{t_n, k}(y, x) :\; (y, x) \in \Sigma_\beta^\intercal \times [w]\} + \frac{\epsilon}{t_n} \;.
\end{equation*}}

Therefore, taking the $\limsup$ when $n$ goes to $\infty$, it follows that \victor{
\begin{equation*}
\limsup_{n \to \infty} \frac{1}{t_n}\log(\mu_{t_n A}([w])) 
\leq \sup\{F_k(y, x) :\; (y, x) \in \Sigma_\beta^\intercal \times [w]\} \;.
\end{equation*}}

On the other hand, by the properties of the $\sup$, there exists \victor{$(y', x') \in \Sigma_\beta^\intercal \times [w]$, such that
\begin{align*}
\sup\{F_k(y, x) :\; (y, x) \in \Sigma_\beta^\intercal \times [w]\} 
&\geq F_k(y', x') \\
&> \sup\{F_k(y, x) :\; (y, x) \in \Sigma_\beta^\intercal \times [w]\} - \epsilon \;.
\end{align*}}

Then, by continuity of the map $F_k$, there exists $N_2 \geq N_1$ such that for any $n \geq N_2$ and each pair $(\widetilde{y}, \widetilde{x}) \in [y']_{|w|+n} \times [x']_{|w|+n}$ is satisfied \victor{
\begin{align}
\sup\{F_k(y, x) :\; (y, x) \in \Sigma_\beta^\intercal \times [w]\} 
&\geq F_k(\widetilde{y}, \widetilde{x}) \nonumber \\
&> \sup\{F_k(y, x) :\; (y, x) \in \Sigma_\beta^\intercal \times [w]\} - 2\epsilon \label{inequality-F_k-sup} \;.
\end{align}}

\victor{Moreover,} following a similar procedure to the one appearing in \eqref{inequality-F_k}, but taking $\inf$ instead of $\sup$, by monotonicity of the measure $\mu_{t_n A^\intercal} \times \mu_{t_n A}$ we can guarantee \victor{the existence} of $N_3 \geq N_2$ such that for any $n \geq N_3$ is satisfied
\[
\mu_{t_n A}([w]) \geq e^{t_n \inf\{F_k(\widetilde{y}, \widetilde{x}) :\; (\widetilde{y}, \widetilde{x}) \in [y']_{|w|+N_3} \times [x']_{|w|+N_3}\} - \epsilon} \mu_{t_n A}([x']_{|w|+N_3}) \;.
\] 

Therefore, taking $\log$ in both of the sides of the inequality and multiplying by $\frac{1}{t_n}$, we obtain that
\[
\frac{1}{t_n}\log(\mu_{t_n A}([w])) \geq \inf\{F_k(\widetilde{y}, \widetilde{x}) :\; (\widetilde{y}, \widetilde{x}) \in [y']_{|w|+N_3} \times [x']_{|w|+N_3}\} - \frac{\epsilon}{t_n} \;.
\]

Finally, taking the $\liminf$ when $n$ goes to $\infty$, by \eqref{inequality-F_k-sup}, it follows that \victor{
\begin{align*}
\liminf_{n \to \infty} \frac{1}{t_n}\log(\mu_{t_n A}([w])) 
&\geq \inf\{F_k(\widetilde{y}, \widetilde{x}) :\; (\widetilde{y}, \widetilde{x}) \in [y']_{|w|+N_3} \times [x']_{|w|+N_3}\} \\
&\geq \sup\{F_k(y, x) :\; (y, x) \in \Sigma_\beta^\intercal \times [w]\} - 2\epsilon \label{inequality-F_k-sup}\;.
\end{align*}}

Hence, taking the limit when $\epsilon$ goes to $0$, we conclude the proof of this lemma.
\end{proof}

\victor{Before proving Theorem \ref{large-deviations-theorem},} which is the main result in this paper, it is necessary to prove some technical lemmas that establish the relation between the natural extension of the so called rate function, \victor{denoted by $\widehat{I}$}, and the uniform limit of the sequence of functions $(F_k)_{k \geq 1}$. The first one of \victor{those lemmas} is the following.

\begin{lemma}
\label{lemma-upper-sc-function}
Consider $\beta > 1$ and \victor{assume that $\Sigma_\beta$ satisfies the specification property.} \victor{Let $W \in \mathcal{H}_\theta(\widehat{\Sigma_\beta})$ be an involution kernel associated to $A \in \mathcal{H}_\theta(\Sigma_\beta)$ and assume that $I(x) > -\infty$ for each $x \in \Sigma_\beta$. Then, for each $(y, x) \in \widehat{\Sigma_\beta}$,} the natural extension $\widehat{I}$ of the map $I$ defined in \eqref{upper-sc-function} satisfies the following expression 
\begin{align*}
\widehat{I}(y, x) 
=& \lim_{k \to \infty} \Bigl(-(W - \widehat{V} - \widehat{V^\intercal})(\widehat{\sigma}^k(y, x)) + W(y, x) - V(x) - V^\intercal(y)\\
&- \sum_{j = 0}^{k - 1}(V^\intercal \circ \sigma - V^\intercal - A^\intercal + m)(\tau_{x, j + 1}(y)) \Bigr) \;.
\end{align*} 
\end{lemma}
\begin{proof}
Since $V$ is a calibrated sub-action associated to the potential $A$, it follows that for each pair \victor{$(y, x) \in \widehat{\Sigma_\beta}$ the sequence 
\[
k \mapsto -\sum_{j = 0}^{k - 1}(V^\intercal \circ \sigma - V^\intercal - A^\intercal + m)(\tau_{x, j + 1}(y)) \;,
\]}
is decreasing and \victor{bounded from below,} which implies convergence.

Besides that, for any pair \victor{$(y, x) \in \widehat{\Sigma_\beta}$} is satisfied
\begin{align*}
\widehat{I}(y, x)
=& \lim_{k \to \infty} -\sum_{j = 0}^{k - 1} (\widehat{V} \circ \widehat{\sigma} - \widehat{V} - \widehat{A} + m)(\widehat{\sigma}^j(y, x)) \\
=& \lim_{k \to \infty} \Bigl(-(W - \widehat{V} - \widehat{V^\intercal})(\widehat{\sigma}^k(y, x)) + W(y, x) - \widehat{V}(y, x) - \widehat{V^\intercal}(\widehat{\sigma}^k(y, x)) \\
& \ \ - \sum_{j = 0}^{k - 1} (-\widehat{A^\intercal} + m)(\widehat{\sigma}^{j+1}(y, x)) \Bigr) \\
=& \lim_{k \to \infty} \Bigl(-(W - \widehat{V} - \widehat{V^\intercal})(\widehat{\sigma}^k(y, x)) + W(y, x) - \widehat{V}(y, x) - \widehat{V^\intercal}(y, x) \\
& \ \ - \sum_{j = 0}^{k - 1} (\widehat{V^\intercal} \circ \widehat{\sigma}^{-1} - \widehat{V^\intercal} - \widehat{A^\intercal} + m)(\widehat{\sigma}^{j+1}(y, x)) \Bigr) \\
=& \lim_{k \to \infty} \Bigl(-(W - \widehat{V} - \widehat{V^\intercal})(\widehat{\sigma}^k(y, x)) + W(y, x) - V(x) - V^\intercal(y)\\
& \ \ - \sum_{j = 0}^{k - 1}(V^\intercal \circ \sigma - V^\intercal - A^\intercal + m)(\tau_{x, j + 1}(y)) \Bigr) \;.
\end{align*} 

Such as we wanted to prove.
\end{proof}

In the following lemma we will check the behavior of the sequences $(\widehat{\sigma}^k(y, x))_{k \geq 1}$ when \victor{$(y, x) \in \widehat{\Sigma_\beta}$.} More specifically, we verify that it is always possible to find a maximizing measure supported on the accumulation points of \victor{those sequences.} The statement of the result is the following.
\begin{lemma}
\label{maximizing-measure-lemma}
Consider \victor{$(y, x) \in \widehat{\Sigma_\beta}$ and $\widehat{\nu}_\infty$ a weak* accumulation point of} the sequence of \victor{periodic probability measures $(\widehat{\nu}_k)_{k \geq 1}$ on $\widehat{\Sigma_\beta}$} given by
\begin{equation*}
\widehat{\nu}_k := \frac{1}{k}\sum_{j = 0}^{k - 1}\delta_{\widehat{\sigma}^j(y, x)} \;.
\end{equation*}

\victor{Then, $\widehat{\nu}_\infty$ is a $\widehat{\sigma}$-invariant measure on $\widehat{\Sigma_\beta}$. Furthermore, if $\widehat{I}(y, x) > -\infty$, we have 
\[
\int_{\widehat{\Sigma_\beta}} \widehat{A} d\widehat{\nu}_\infty = m \;,
\]}
where $m$ is \victor{the general maximizing value} given by \eqref{general-maximizing-value}. 
\end{lemma}
\begin{proof}
\victor{Note that the $\widehat{\sigma}$-invariance of $\widehat{\nu}_\infty$ follows from a straightforward argument. On the other hand, if} $\widehat{I}(y, x) = I(x) > -\infty$ for any pair \victor{$(y, x) \in \widehat{\Sigma_\beta}$,} given $\epsilon > 0$, there is $k_\epsilon \in \mathbb{N}$ such that for any \victor{pair of natural numbers} $k_1 > k_0 \geq k_\epsilon$ is satisfied
\[
\sum_{j = k_0}^{k_1 - 1}(V \circ \sigma - V - A + m)(\sigma^j(x)) < \epsilon \;,
\]
which is \victor{equivalent to saying} that 
\[
\frac{V(\sigma^{k_1}(x)) - V(\sigma^{k_0}(x)) - \epsilon}{k_1 - k_0} + m < \frac{1}{k_1 - k_0}\sum_{j = k_0}^{k_1 - 1}A(\sigma^j(x)) \;.
\]

The above implies that
\[
m \leq \liminf_{k \to \infty} \frac{1}{k - k_0}\sum_{j = k_0}^{k - 1}A(\sigma^j(x)) \;.
\] 

In particular, taking $\epsilon > \sup(I)$, it follows that \victor{
\begin{align}
m 
&\leq \liminf_{k \to \infty} \frac{1}{k}\sum_{j = 0}^{k - 1}A(\sigma^j(x)) \nonumber \\
&\leq \liminf_{k \to \infty} \frac{1}{k}\sum_{j = 0}^{k - 1}\widehat{A}
(\widehat{\sigma}^j(y, x)) 
= \liminf_{k \to \infty} \int_{\widehat{\Sigma_\beta}} \widehat{A} d\widehat{\nu}_k 
\label{max-measure-I} \;.
\end{align}}

On the other hand, for each $k \in \mathbb{N}$ is satisfied \victor{
\begin{align}
\limsup_{k \to \infty} \int_{\widehat{\Sigma_\beta}} \widehat{A} d\widehat{\nu}_k 
&= \limsup_{k \to \infty} \frac{1}{k}\sum_{j = 0}^{k - 1}\widehat{A}
(\widehat{\sigma}^j(y, x)) \nonumber \\
&= \limsup_{k \to \infty} \frac{1}{k}\sum_{j = 0}^{k - 1}A(\sigma^j(x)) 
\leq m \label{max-measure-II} \;.
\end{align}}

Therefore, \victor{by \eqref{max-measure-I}, \eqref{max-measure-II} and the asumption that $\widetilde{\nu}_\infty$ is an accumulation point of the sequence $(\widetilde{\nu}_k)_{k \geq 1}$ in the weak* topology, it follows that
\[
\int_{\widehat{\Sigma_\beta}} \widehat{A} d\widehat{\nu}_\infty = \lim_{k \to \infty} \int_{\widehat{\Sigma_\beta}} \widehat{A} d\widehat{\nu}_k = m \;.
\]}
\end{proof}
\begin{remark}
\label{unique-maximizing-measure}
Observe that under the assumption that the potential $A$ has a unique maximizing measure \victor{which is equivalent to saying that $\widehat{A}$ has a unique measure $\widehat{\mu}_{\max}$ such that $\int_{\widehat{\Sigma_\beta}} \widehat{A} d\widehat{\mu}_{\max} = m$,} it follows that $\widehat{\nu}_\infty = \widehat{\mu}_\infty = \widehat{\mu}_{\max}$, where the measure $\widehat{\mu}_\infty$ is the one obtained in Proposition \ref{natural-extension-proposition}. Therefore, in that case the measure $\widehat{\nu}_\infty$ results in the natural extension of the unique accumulation point $\mu_\infty$ of the family of Gibbs states $(\mu_{t A})_{t > 1}$ to the bilateral $\beta$-shift.
\end{remark}

Now we are able to present the proof of the main result of this paper, \victor{which guarantees} that is satisfied a first level large deviations principle in the setting of $\beta$-shifts via involution kernel and transpose potentials. \victor{We point out that we only prove the large deviations principle in the case when the rate function takes real values. The above, because the extension to rate functions defined on $\overline{\mathbb{R}}$ is a trivial consequence of the case studied here.} The statement of the main result of this paper is the following.

\begin{theorem}
\label{large-deviations-theorem}
Consider $\beta > 1$ and \victor{assume that $\Sigma_\beta$ satisfies the specification property.} Fix a strictly increasing sequence $(t_n)_{n \geq 1}$ satisfying \eqref{calibrated-sub-actions} and assume that the potential $A$ has a unique maximizing measure $\mu_{\infty}$. Suppose that \victor{$A \in \mathcal{H}_\theta(\Sigma_\beta)$ and let $I$ be a rate function such as appears defined in \eqref{upper-sc-function} such that $I(x) > -\infty$ for each $x \in \Sigma_\beta$. Then, for any cylinder $[w] \subset \Sigma_\beta$} is satisfied
\begin{equation*}
\lim_{n \to \infty}\frac{1}{t_n}\log(\mu_{t_n A}([w])) = \sup\{I(x) :\; x \in [w]\} \;.
\end{equation*} 
\end{theorem}
\begin{proof}
As a consequence of Lemma \ref{large-deviations-lemma-F_k}, Lemma \ref{lemma-upper-sc-function} and the equation in \eqref{F_k-function}, it is enough to prove that for any pair \victor{$(y, x) \in \widehat{\Sigma_\beta}$} is satisfied
\begin{equation}
\label{extension-rate-function}
\widehat{I}(y, x) = \lim_{k \to \infty} F_k(y, x) \;.
\end{equation}

\victor{Furthermore, since} $V^\intercal$ is a sub-action associated to the potential $A^\intercal$, it follows that for each pair \victor{$(y, x) \in \widehat{\Sigma_\beta}$} the sequence \victor{ 
\[
k \mapsto -\sum_{j = 0}^{k - 1} (V^\intercal \circ \sigma - V^\intercal - A^\intercal + m)(\tau_{x, j + 1}(y)) \;,
\]}
is decreasing. The above, joint with Lemma \ref{lemma-upper-sc-function}, implies existence of the limit \victor{
\[
\lim_{k \to \infty} (W - \widehat{V} - \widehat{V^\intercal})(\widehat{\sigma}^k(y, x)) \;.
\] }

Hence, in order to prove \eqref{extension-rate-function}, we only need to check the following limit
\begin{equation*}
\lim_{k \to \infty} (W - \widehat{V^\intercal} - \widehat{V})(\tau_{x, k}(y), \sigma^k(x)) = \gamma \;.
\end{equation*}

Actually, \victor{the existence of the above limit is a} consequence of Lemma \ref{mean-integral-lemma}, Lemma \ref{maximizing-measure-lemma} and Remark \ref{unique-maximizing-measure}. The above, because we are assuming that the potential $A$ has a unique maximizing measure $\mu_\infty$. 
\end{proof}

\victor{
\section*{Acknowledgments}
The author would to thank to INCTMat and the Francisco Jos\'e de Caldas Fund by the financial support during part of the development of this paper.
}

\end{document}